\documentclass[11pt,a4paper,final]{article}
\usepackage[body={14.6true cm,  22true cm},
            headheight=1.5 em]{geometry}
\usepackage{amsmath}
\usepackage{graphics}
\usepackage{graphicx}
\usepackage{subfigure}
\usepackage{indentfirst}
\usepackage{epsfig}
\usepackage{hyperref}
\usepackage{tikz}
\usepackage{tabu}
\usepackage{multicol}
\usepackage{amsfonts}
\usepackage{float}
\usepackage{enumerate}
\newenvironment{proof}{\noindent \textbf{Proof.}
}{\unskip\nobreak\hfill\nobreak$\square$\vskip 0.5\baselineskip}

\newtheorem{theorem}{Theorem}[section]
\newtheorem{lemma}{Lemma}[section]
\newtheorem{proposition}{Proposition}[section]

 \def\square{\hbox{\vrule height8pt depth0pt
\vbox{\hrule width7.2pt\vskip7.2pt\hrule width7.2pt}\vrule height8pt
depth0pt}\smallskip}

\date{}

\begin{document}

\title{The energy change of the complete multipartite graph\footnote{Partially supported  by NSFC project No. 11271288. Email:shan\_haiying@tongji.edu.cn (Hai-Ying Shan), changxiang-he@163.com (Chang-Xiang He, Corresponding author)  }}
\author{\small Hai-Ying Shan$^{1}$, Chang-Xiang He$^{2}$, Zhen-Sheng Yu$^{2}$,
\\
\small 1.  School of Mathematical Sciences, Tongji University, Shanghai, 200092, China\\ 
\small 2.  College of Science, University of Shanghai for Science
and Technology, Shanghai, 200093, China}
\date{}
\maketitle

 {\small\bf Abstract:}\,\,\,{\small The energy of a graph is
defined as the sum of the absolute values of all eigenvalues of the
graph. Akbari et al.  \cite{S. Akbari} proved that for a complete multipartite graph  $K_{t_1 ,\ldots,t_k}$,  if   $t_i\geq 2 \ (i=1,\ldots,k)$,  then deleting any edge will  increase the energy. A  natural  question is how the energy changes when $\min\{t_1 ,\ldots,t_k\}=1$. In this paper, we will answer this question and  completely determine how the energy  of a complete multipartite graph changes when  one edge is removed.
\vspace{2mm}


\section{Introduction}

 Let $G = (V, E)$ be a simple connected graph with vertex set $V =
\{v_1, v_2,\cdots, v_n\}$ and edge set $E$.
 The {\it adjacency
matrix} of $G$, $A(G) = (a_{ij})$, is an $n \times n$ matrix, where
$a_{ij} = 1$ if $v_i$ and $v_j$ are adjacent and $a_{ij} = 0$,
otherwise. Thus $A(G)$ is a real symmetric matrix with zeros on the
diagonal, and all eigenvalues of $A(G)$ are real.   The {\it
characteristic polynomial}   $\det(xI-A(G))$ of the adjacency matrix $A(G)$ of a graph $G$ is also called the   characteristic polynomial of $G$, denoted by $\Phi(G,x)$ or $\Phi(G)$. The {\it eigenvalues} of graph $G$  are
the eigenvalues of $A(G)$, written as  $\lambda_1(G)\geq \lambda_2(G)\geq\ldots \geq\lambda_n(G)$.
The energy of $G$, denoted by $\mathcal{E}(G)$, is defined \cite{Gutman1, Gutman2} as
$\mathcal{E}(G) =\sum\limits_{i=1}^n \vert\lambda_i(G)\vert$.

For the polynomial $f(x)$, if all the roots of $f(x)=0$ are real, we also define the energy of $f(x)$ as  the sum of the absolute values of its roots, denoted by $\mathcal{E}(f)$.

We denote a complete multipartite graph with $k\geq2$ parts by $K_{t_1 ,\ldots,t_k}$, where $t_i\ (i=1,\ldots,k)$ is the number of vertices in the $i^{th}$ part of the graph, and we write the $i^{th}$ part as $t_i$-part.

One area in the study of graph energy, called graph energy change is to understand how graph energy changes when a subgraph is deleted. It becomes especially interesting when the subgraph is just an edge. As we know the energy of a graph may increase, decrease, or remain the same when an edge is deleted.  For more details see \cite{J. Day} and the references therein.

Akbari, Ghorbani
and Oboudi  \cite{S. Akbari} (see Theorem 4) proved that
 for any complete multipartite graph $K_{t_1 ,\ldots,t_k}$   with $k \geq 2, t_i \geq 2$,
then $ \mathcal{E}(K_{t_1 ,\ldots,t_k} - e)>\mathcal{E}(K_{t_1,\ldots,t_k} ) $ for any edge $e$.
Then a natural  question is how the energy changes when $\min\{t_1 ,\ldots,t_k\}=1$. In this paper, we will answer this question and  completely determine how the energy  of a complete multipartite graph changes when  one edge is deleted. Our main result is
\begin{theorem}\label{thm:main} Let  $e$ be an edge between the $t_i$-part and $t_j$-part of $K_{t_1,\ldots,t_k}$.  Then
\begin{enumerate}[(1).]
\item  For $k \geq 4$,  if $t_i=t_j=1$, then  $\mathcal{E}(K_{t_1,\ldots,t_k}-e )<\mathcal{E}(K_{t_1,\ldots,t_k})$, otherwise,  $\mathcal{E}(K_{t_1,\ldots,t_k}-e)>\mathcal{E}(K_{t_1,\ldots,t_k})$.
\item For  $k=3$, if $t_i+t_j \leq 3$, then  $\mathcal{E}(K_{t_1,\ldots,t_k}-e)<\mathcal{E}(K_{t_1,\ldots,t_k})$, otherwise, $\mathcal{E}(K_{t_1,\ldots,t_k}-e)>\mathcal{E}(K_{t_1,\ldots,t_k})$.
\item For $k=2$, if $\min\{t_i,t_j\}=1$, then  $\mathcal{E}(K_{t_1,\ldots,t_k}-e)<\mathcal{E}(K_{t_1,\ldots,t_k})$, otherwise, $\mathcal{E}(K_{t_1,\ldots,t_k}-e)>\mathcal{E}(K_{t_1,\ldots,t_k})$.
\end{enumerate}
\end{theorem}

This paper is organized as follows. In Section 2, we will give a  generalization of Theorem 4 in \cite{S. Akbari} and some   results which will be needed in the next two sections.  In the third section, we will determine how the energy of a complete multipartite graph, with   at least four parts, changes when an edge is removed.   In the last section, we will characterize how the energy  of a complete tripartite graph changes when  an  edge is deleted.

\section{Preliminaries}

We begin this section with the Interlacing Theorem. By Perron-Frobenius theory, the largest eigenvalue of a connected graph goes down when one removes an edge or a vertex. Interlacing also gives more information about what happens with the $i^{th}$ largest eigenvalues.
 \begin{lemma}\label{Intelacing}(Interlacing) If $G$ is a graph on $n$ vertices with eigenvalues $\lambda_1(G) \geq \ldots \geq \lambda_n(G)$ and $H$ is an induced subgraph on $m$ vertices with eigenvalues $\lambda_1(H) \geq \ldots \geq \lambda_m(H)$, then for $i=1,\ldots, m,$
$$\lambda_i(G)\geq \lambda_i(H)\geq\lambda_{n-m+i}(G).$$
\end{lemma}
In the next two sections, we will use $\lambda_2(G)\geq \lambda_2(H)$ (where $H$ is an induced subgraph of $G$) many times.

As known, equitable partition   represents a powerful tool in spectral graph theory. In this paper we also should use this powerful tool to simplify our calculation.

Given a graph $G$, the partition $V(G)=V_1\dot\cup V_2\dot\cup\ldots \dot\cup V_k$ is an equitable partition if every vertex in $V_i$ has the same number of neighbours in $V_j$, for all $i,\ j\in \{1,2,\ldots,k\}$.
%In a complete multipartite graph the usual colouring gives rise to an equitable partition in which the cells are the colour classes.
Suppose $\Pi$ is an equitable partition $V(G)=V_1\dot\cup V_2\dot\cup\ldots \dot\cup V_k$ and that each vertex in $V_i$ has $b_{ij}$ neighbours in $V_j$ $(i,\ j\in \{1,2,\ldots,k\})$. The matrix $(b_{ij})$ is called the quotient matrix of $\Pi$, denoted by $B_{\Pi}$.  The largest  eigenvalue of $B_{\Pi}$ is also the spectral radius of $G$ (see     \cite{Cvetkovic},    Corollary 3.9).
 In order to determine the spectral radius of graph $G$, we can calculate the largest root of the characteristic polynomial of  one of its quotient matrices, which   has a lower degree.

For convenience, in this paper, we use   $\lambda(G)$ and ${\bf x}$, respectively, to denote the spectral radius and the corresponding unit eigenvector of the adjacency matrix of $G$.
Suppose $V_i$ is the  $t_i$-part of $K_{t_1,t_2\ldots,t_k}$,  then $V_1\cup V_2\cup\cdots\cup V_k  $ is an equitable partition.  Unless otherwise specified, the  cells of equitable partition of $K_{t_1,t_2\ldots,t_k}$  are $V_1,\ V_2,\ldots, V_k  $.
% The  cells of equitable partition of $K_{t_1,t_2\ldots,t_k}$, unless otherwise specified, are the colour classes.
Obviously,   vertices in the same part $V_i$ have equal ${\bf x}$-components, denoted by $x_i$.

Akbari, Ghorbani
and Oboudi  ( see Theorem 4 in \cite{S. Akbari} ) proved that
 for any complete multipartite graph $K_{t_1 ,\ldots,t_k}$   with $k \geq 2, t_i \geq 2$,
then $ \mathcal{E}(K_{t_1 ,\ldots,t_k} - e)>\mathcal{E}(K_{t_1,\ldots,t_k} ) $ for any edge $e$.
Using the idea of Akbari, Ghorbani
and Oboudi, we get   a  generalization of this result.

\begin{theorem}\label{thm:two} Let $S$ be a non-empty edge subset of the complete multipartite graph $G=K_{t_1 ,\ldots,t_k}$ and $H$ be
the corresponding subgraph induced by $S$.  Let  $V_i$ be the $i^{th}$  part of $G$, and $U_i=V(H)\cap V_i$ $(1\leq i\leq k)$. If $ |V_i|\geq 2\lambda(H) |U_i|$ holds for any $ i $, then we have $\mathcal{E}(G-S)>\mathcal{E}(G)$.
\end{theorem}
\begin{proof} Let $A$ and $B$ be the adjacency matrices of $G$ and $G-S$, respectively. We may assume that $B =A-C$, where $C$ is the adjacency matrix of the spanning subgraph of $G$ with only   edges in $S$. Let ${\bf x}$ be the Perron vector of $A$.

Since each part of $G$ is a cell of an equitable partition of $G$, the vertices of each part have the same corresponding entries in  ${\bf x}$.   By the Rayleigh-Ritz theorem and $|V_i| \geq 2\lambda(H) |U_i|$ for $1 \leq i \leq k$, we have
$${\bf x}^TC{\bf x}= {\bf y}^T A(H) {\bf y} \leq  \lambda(H){\bf y}^T{\bf y} \leq  \frac{1}{2}{\bf x}^T{\bf x}= \frac{1}{2},$$
where ${\bf y}$ is the subvector of ${\bf x}$ indexed by vertices in $H$.

% the last inequality holds because all ends are in different cells of the  equitable partition  and each cell  has at least two elements.

Thus,
$$\lambda(B)\geq {\bf x}^TB{\bf x}={\bf x}^TA{\bf x}-{\bf x}^TC{\bf x}\geq \lambda(A)-\frac{1}{2}.$$

Suppose $e=(u,v) \in S$ and $u \in U_i$, $v \in U_j$.
Since $S$ is nonempty, $\lambda(H) \geq 1$. So $|V_i|\geq 2|U_i|>|U_i|$,$|V_j|\geq 2|U_j|>|U_j|$. Let $u' \in V_i-U_i$ and $v' \in V_j-U_j$. Then $ P_4 =uv'u'v$ is an induced subgraph of $G−e$. Therefore, by the Interlacing theorem,
$$\lambda_2(B)\geq \lambda_2(P_4)\approx 0.618.$$
Thus $$ \mathcal{E}(
G - S)\geq 2(\lambda(B)+\lambda_2(B))>2\lambda(A)=\mathcal{E}(G). $$
\end{proof}

 Obviously, Theorem \ref{thm:two} generalizes  Theorem 4 in \cite{S. Akbari}. Observe that if $U_i=\emptyset$, the condition $ |V_i|\geq 2\lambda(H) |U_i|=0$ holds trivially whether $|V_i|$ is 1 or not, so  the complete multipartite graph in the above theorem  needs not  be  1-part free.

  Theorem \ref{thm:two} immediately implies that deleting any edge between non-1-parts of the complete multipartite graph will increase the energy.

However, deleting one edge between two 1-parts of complete multipartite graph will decrease the energy. If  $K_{t_1 ,\ldots,t_k}$ has  two 1-parts, without loss of generality, we assume $t_1=t_2=1$, and  $e$ is the edge between these 1-parts, then $$\mathcal{E}(K_{1,1,t_3 ,\ldots,t_k}-e)=\mathcal{E}(K_{2,t_3 ,\ldots,t_k})=2\lambda (K_{2,t_3 ,\ldots,t_k})<2\lambda (K_{1,1,t_3 ,\ldots,t_k})=\mathcal{E}(K_{1,1,t_3 ,\ldots,t_k}).$$

In order to completely determine how the energy  of complete multipartite graph changes when one edge is removed, we only need to consider    the deleted edge is between a 1-part and a non-1-part.
So in the next we assume that the considered  complete multipartite graph  is $K_{1,i,t_3\ldots,t_k}$ (where $i\geq 2$) and the deleted edge    is between 1-part and $i$-part.

Without loss of generality, we assume that   ${\bf x}$-components of the ends of the deleted edge are $x_1$ and $x_2$, respectively.

The following lemma is a starting point of our discussions.

\begin{lemma}\label{lem:rey} Let  ${\bf x}$ be a perron vector of complete multipartite graph $G$. Let $e=uv$ be an edge of $G$ and the   corresponding entries  in  ${\bf x}$ be $x_1$ and $x_2$, respectively. If there exists some constant $a$ such that $\lambda_2(G-e)>a$ and $x_1^2+x_2^2\leq a$, then $\mathcal{E}(G-e)>\mathcal{E}(G)$.
\end{lemma}
\begin{proof} Let $A(G-e)=A(G)-C$, where $C$ is the adjacency matrix of the spanning subgraph of $G$ with only one edge $e$. Then  $${\bf x}^TC{\bf x}=2x_{1}x_{2}\leq x_1^2+x_2^2\leq a.$$
By the Rayleigh-Ritz theorem,
$$\lambda (G-e)\geq {\bf x}^TA(G-e){\bf x}={\bf x}^TA(G){\bf x}-{\bf x}^TC{\bf x}\geq \lambda (G)-a.$$
Because $\lambda_2(G-e)>a$, we  arrive at $$ \mathcal{E}(G-e)\geq 2(\lambda(G-e)+\lambda_2(G-e))>2\lambda(G)=\mathcal{E}(G). $$
\end{proof}

Next we will give a  lower  bound on the spectral radius of complete multipartite graph $K_{1,i,t_3 ,\ldots,t_k}$ which will be used in the calculation in the subsequent sections.
\begin{lemma}\label{lembound}
Let $G=K_{1,i,t_3 ,\ldots,t_k}$ be a complete multipartite graph with order $n$. We have:

\begin{enumerate}[(1).]
\item If  $k\geq 3$, then $\lambda(G)>\sqrt{(n-i)(i+1)}.$
\item  In particular, if $2\leq i\leq n-5$ and $\max\{t_3,\ldots,t_k\}=1$, then $\lambda(G)>n-i+0.67$ holds.
\end{enumerate}
\end{lemma}

\begin{proof}
 (1) The characteristic polynomial of the quotient matrix of $K_{1,i,n-i-1}$ is $$f(x)=x^3-\big ((n-i)(i+1)-1 \big )x-2(n-i-1)i.$$
It is easy to see that $\lambda(K_{1,i,n-i-1})>\sqrt{(n-i)(i+1)}$. Note that $K_{1,i,n-i-1}$ is a subgraph of $G$, so that, $\lambda(G)\geq \lambda(K_{1,i,n-i-1})>\sqrt{(n-i)(i+1)}$.

(2) If $\max\{t_3,\ldots,t_k\}=1$, then
  $$Q=\left(\begin{array}{rr}
0 & n-i \\
i & n-i-1 \\
\end{array}\right)$$ is a quotient matrix of $G$, so $\lambda(G)$ is the largest root of $\phi(Q,x)= x^2-(n-i-1)x-i(n-i)$.

If $2\leq i \leq n-5$, we have
$$
\begin{aligned}
\phi(Q,n-i+0.67)=&i^{2} -{\left(i - 1.67\right)} n - 1.67 \, i + 1.1189\\
\leq  & i^{2}  -{\left(i - 1.67\right)}  (i+5) - 1.67 \, i + 1.1189\\
= &9.4689-5i<0.
\end{aligned}
$$
Therefore, $\lambda(G)>n-i+0.67$.

\end{proof}

The following lemma provides some sufficient  (but not necessary)  conditions for $ \mathcal{E}(G-e)>\mathcal{E}(G ) $, and   is also a key tool which will  be widely used in the sequel proofs.
\begin{lemma}\label{lem:cond} Let $G=K_{1,i,t_3 ,\ldots,t_k}$ be a complete $k$-partite graph ($k\geq 3$) with order $n$, and $e$ be an edge between 1-part and $i$-part. Suppose that $a$ is a positive constant  and $\lambda_2(G-e)>a$.  If one of the following  holds:
\begin{enumerate}[(1).]
\item  $\frac{2i+1}{i(i+2)}<a<1$ and $f_{a}(n,i)=n\left( a i^2-2 (1-a)i-1 \right)-ai^3+(1-a)i^2-(a-2)i>0$,
\item  $\frac{2(n-1)}{\lambda^2+n-1}<a$, where $\lambda=\lambda(G)$,
\end{enumerate}
then   $ \mathcal{E}(G-e)>\mathcal{E}(G ) $.
\end{lemma}
\begin{proof}
  By Lemma \ref{lem:rey}, it suffices to prove $x_1^2+x_2^2\leq a$.

(1).
 Now $a>\frac{2i+1}{i(i+2)} $, which means $ai^2-2(1-a)i-1>0$. Combining this with $f_a(n,i)>0$, i.e., $n\geq \frac{ai^3-(1-a)i^2+(a-2)i}{ai^2-2(1-a)i-1}$, we can get
$\lambda^2(G)>(n-i)(i+1)\geq\frac{(n-1)(1-a)i}{ai-1},$  which yields  that
$$\frac{(n-1)i}{\lambda^2(G)}<\frac{ai-1}{1-a}.$$
From the eigenvalue equation of $G$, we have
$\lambda(G) x_1=ix_2+t_3x_3+\ldots+t_kx_k $.
  Applying the Cauchy-Schwarz inequality, we see that
 $$\lambda^2(G)x_1^2\leq (i+t_3+\ldots+t_k)(ix_2^2+t_3x_3^2+\ldots +t_kx_k^2)=(n-1)(ix_2^2+m)=(n-1)(1-x_1^2) ,$$
where $m=t_3x_3^2+\ldots+ t_kx_k^2$. Hence  $x_1^2\leq\frac{(n-1)}{\lambda^2(G)} \left( ix_2^2+m \right)\leq \frac{ai-1}{1-a}x_2^2+\frac{a}{1-a}m$. This shows that
$$x_1^2+x_2^2\leq \frac{a}{1-a}((i-1)x_2^2+ m)=\frac{a}{1-a}(1-x_1^2-x_2^2),$$ which implies that $x_1^2+x_2^2\leq a$ holds.

(2). By considering eigenvalue equations $\lambda x_1=ix_2+t_3x_3+\ldots+t_kx_k $ and $\lambda x_2=x_1+t_3x_3+\ldots+t_kx_k $,  we find $x_2=(\frac{\lambda+1}{\lambda+i})x_1$.  From $\lambda^2x_1^2\leq (n-1)(1-x_1^2) $, we have
$$x_1^2\leq \frac{n-1}{\lambda^2+n-1}.$$

Therefore,
$$x_{1}^2+x_{2}^2=\left(1+\bigg(\frac{\lambda+1}{\lambda+i}\bigg)^2\right)x_1^2\leq \left(1+\bigg(\frac{\lambda+1}{\lambda+i}\bigg)^2\right)\frac{n-1}{\lambda^2+n-1}< \frac{2(n-1)}{\lambda^2+n-1}<a.$$
\end{proof}

\section{The complete multipartite graph with at least four parts}

In this section, we consider how the energy changes of the complete multipartite graph $K_{1,i,t_3 ,\ldots,t_k}$, where $k\geq 4$, by deleting an edge between 1-part and $i$-part. We will distinguish into two cases  $i\geq 4$ and $i\in\{2,3\}$, and will apply the two methods in Lemma \ref{lem:cond} to compare the energies of $K_{1,i,t_3 ,\ldots,t_k}-e$ and $K_{1,i,t_3 ,\ldots,t_k}$.  Now we consider the case of $i\geq 4$ firstly.
\begin{lemma}\label{lem:four}  If  $k\geq 4$ and $   i\geq 4$, then $\mathcal{E}(K_{1,i,t_3 ,\ldots,t_k}-e)>\mathcal{E}(K_{1,i,t_3 ,\ldots,t_k})$ for any edge $e$ between $1$-part and $i$-part.
\end{lemma}
\begin{proof}
As $k\geq 4$, $K_{1,4,1,1}-e$ is an induced subgraph of $K_{1,i,t_3 ,\ldots,t_k}-e$, by the  Interlacing Theorem  $\lambda_2(K_{1,i,t_3 ,\ldots,t_k}-e)\geq \lambda_2(K_{1,4,1,1}-e)= \sqrt{2}-1>0.4$ holds.

Since $\frac{2i+1}{i(i+2)}$ is a decreasing function for $i$, we have $\frac{2i+1}{i(i+2)} \leq  \frac{3}{8} <0.4 $ for $i \geq 4$.
Now we   use Lemma \ref{lem:cond} by taking $a=0.4$, then
$$5f_{0.4}(n,i)\geq 5f_{0.4}(i+3,i)=3(i^2-5i-5)>0 $$ holds for all $i\geq 6$. Hence, $ \mathcal{E}(K_{1,i,t_3 ,\ldots,t_k}-e)>\mathcal{E}(K_{1,i,t_3 ,\ldots,t_k} ) $ holds when $i\geq6$.

% For $i=4,\ 5$， take $a=0.414$. It is easy to see $\frac{2i+1}{i(i+2)} \leq  \frac{9}{24} <0.4 $
Because $f_{0.414}(n,4)>0$  when $n \geq 12$ and $f_{0.414}(n,5)>0 $ when $n \geq 9$, and these  show  that  $ \mathcal{E}(K_{1,i,t_3 ,\ldots,t_k}-e)>\mathcal{E}(K_{1,i,t_3 ,\ldots,t_k} ) $ holds for $i=4,\ 5$ when $n\geq 12$ and $n\geq 9$, respectively.
With the aid of mathematics software ``SageMath'' \cite{sage}, one can verify the result holds  for $i=4,\  n\leq 11$ and $i=5,\ n\leq 8$.
\end{proof}

The next lemma offers a method   to compare the spectral radius of two complete multipartite graphs with the same order, which will be used in the proof of Lemma \ref{lem:ftt}.

\begin{lemma}(\cite{Dragan})\label{multi} If $n_i-n_j \geq 2$, then $\lambda (K_{n_1,\ldots,n_{i}-1,\ldots,n_{j}+1,\ldots,n_p}) > \lambda (K_{n_1 ,\ldots,n_i,\ldots,n_j,\ldots,n_p})$.
\end{lemma}

\begin{lemma}\label{lem:ftt} $ \mathcal{E}(K_{1,i,t_3 ,\ldots,t_k}-e)>\mathcal{E}(K_{1,i,t_3 ,\ldots,t_k} ) $ for any $k\geq 4$,   $i\in\{2,3\}$ and every edge $e$ between $1$-part and $i$-part.
\end{lemma}
\begin{proof} For short,  we write  $ K_{1,i,t_3 ,\ldots,t_k}$ as $G$.

 By Lemma \ref{lembound}, when $n\geq 8$, if $\max\{t_3,\ldots,t_k\}=1$, $\lambda(G)>n-2.33$ holds, which implies that $\frac{2(n-1)}{\lambda^2+n-1}<\frac{2(n-1)}{(n-2.33)^2+n-1}$.   Note that $\frac{2(n-1)}{(n-2.33)^2+n-1}< 0.357$  when $n\geq 8$. Hence, $\frac{2(n-1)}{\lambda^2+n-1}< 0.357$

If $\max\{t_3,\ldots,t_k\}\geq 2$, say $t_3\geq2$, then $K_{1,i,t_3,n-i-t_3-1}$ is a subgraph of $G$, so
  $$\lambda(G)\geq \lambda(K_{1,i,t_3,n-i-t_3-1})\geq \lambda(K_{1,i,2,n-i-3})\geq \lambda(K_{1,2,2,n-5}),$$
  the last two inequalities follow  from Lemma \ref{multi}. Note that $\lambda(K_{1,2,2,n-5}) $
   is the largest root of $ g(x)=x^4-(5n-17)x^2-8(2n-9)x-6(2n-10)$ which is the characteristic polynomial of its equitable matrix. It is easy to check that $\tau(g)>\sqrt{5n-7 }$ when $n\geq 8$, which means  $\lambda(G)>\sqrt{5n- 7 }$.
 And thus,  $\frac{2(n-1)}{\lambda^2+n-1}< \frac{2(n-1)}{6n-8}< 0.357$, when $n\geq 8$.

 On the other hand, $\lambda_2(G-e)\geq \lambda_2(K_{1,2,1,1}-e) > 0.357$.  Now we use (2) of Lemma \ref{lem:cond} by taking $a=0.357$, we have $ \mathcal{E}(G-e)>\mathcal{E}(G) $.

With the aid of mathematics software ``SageMath'' \cite{sage}, one can verify the result holds   for $G$  when $n\leq 7$.
   % From the Table 1, we can see the desired result holds for $G_2$ and $G_3$ when $n\leq 28$ and $n\leq 13$, respectively.
   \end{proof}

The following proposition is an immediate result from Lemma \ref{lem:four} and Lemma \ref{lem:ftt}.
\begin{proposition}\label{pro:mult} $\mathcal{E}(K_{1,i,t_3,\ldots,t_k}-e)>\mathcal{E}(K_{1,i,t_3,\ldots,t_k})$ holds for any $k\geq 4,\ i\geq 2$ and $e$   between 1-part and $i$-part.
\end{proposition}

\section{Complete tripartite graph}

In this section, we will focus on the energy change of the complete tripartite graph $K_{1,i,n-i-1}$. We distinguish into two cases: $ 4\leq i\leq n-3$ and $i\in\{2,3,n-2\}$. The proof of the first   case  is similar to the proof of Lemma \ref{lem:four}. But for the case  $i\in\{2,3,n-2\}$, it is almost impossible to use the former  method, so we will give another new energy comparison method.

\begin{lemma}\label{lem:three}
If $ 4\leq i\leq n-3$, then  $\mathcal{E}(K_{1,i,n-i-1}-e)>\mathcal{E}(K_{1,i,n-i-1})$,  for any edge $e$    between 1-part and $i$-part.
\end{lemma}
\begin{proof}
If $i\geq 8$, with the similar manner of Lemma \ref{lem:four},
 $K_{1,5,2}-e$ is an induced subgraph of $K_{1,i,n-i-1}-e$, by the  Interlacing Theorem, $\lambda_2(K_{1,i,n-i-1}-e)\geq \lambda_2(K_{1,5,2}-e) > \frac{11}{30}$.

Taking $a=\frac{11}{30}$, we find $f_{a}(n,i)$ is a strictly increasing function for $n$. Since $n\geq i+3$,  we easily have that
$$30 f_{a}(n,i)\geq 30 f_{a}(i+3,i)= 14 i^2 - 95i - 90>0 $$ holds when $i\geq 8$. Since $\frac{2i+1}{i(i+2)} \leq  \frac{17}{80} <a $ for  $i\geq 8$, by (1) of Lemma \ref{lem:cond}, we have $ \mathcal{E}(K_{1,i,n-i-1  }-e)>\mathcal{E}(K_{1,i,n-i-1  } ) $.

When $4\leq i\leq 7$, we take $a= 0.36 < \lambda_2(K_{1,4,2}-e)$.  Lemma \ref{lembound} provides that
$\frac{2(n-1)}{\lambda^2+n-1}<\frac{2(n-1)}{(n-i)(i+1)+n-1}< \lambda_2(K_{1,4,2}-e)$ for $n \geq 35$. Since $\lambda_2(K_{1,i,n-i-1}-e)>\lambda_2(K_{1,4,2}-e)$, $\mathcal{E}(K_{1,i,n-i-1}-e)>\mathcal{E}(K_{1,i,n-i-1})   $ follows from (2) of Lemma \ref{lem:cond}  .

With the aid of mathematics software ``SageMath'' \cite{sage}, one can verify the result holds for $n \leq 34$.
\end{proof}

Next we will consider how the energy changes of $K_{1,n-2,1}$, $K_{1,2,n-3}$ and $K_{1,3,n-4}$ by deleting one edge between the first two parts.  For convenience, we use $\tau(f)$ to denote the largest real root of the equation  $f(x)=0$ if it exists.

The following is a lemma about the largest root of equation which will be used in the proof of our last lemma.
\begin{lemma}\label{lem:roots} Let $f(x)=x^4+a x^2 + b x +c $ and $g(x)=x^{6} + 8 a x^{4} + 16 (a^{2} - 4 c) x^{2} - 64 b^{2}$. If all roots of the equation  $f(x)=0 $ are real,
\begin{enumerate}[(1).]
\item  then $g(x)=0$  has only real roots.
\item in particular, if $f(x)=0$ has  exactly two positive roots, then $\mathcal{E}(f)=\tau(g)$.
\end{enumerate}
\end{lemma}
 \begin{proof} Let $x_1,x_2,x_3,x_4$ be the four real roots of $f(x)=0$, then
 \begin{align}
x_{1} + x_{2} + x_{3} + x_{4}&=0\\
x_{1} x_{2} + x_{1} x_{3} + x_{2} x_{3} + x_{1} x_{4} + x_{2} x_{4} + x_{3} x_{4}&=a\\
x_{1} x_{2} x_{3} + x_{1} x_{2} x_{4} + x_{1} x_{3} x_{4} + x_{2} x_{3} x_{4}&=-b\\
x_{1} x_{2} x_{3} x_{4}& =c.
\end{align}

 (1) Put $y=2(x_1+x_2)$. By formulas (1) and (2), we see that $$x_1x_2+x_3x_4=a+(x_1+x_2)^2=a+\frac{y^2}{4}.$$ On the other hand,
 $$y(x_1x_2-x_3x_4)=2b$$ follows from formulae (1) and (3).
Then $(a+\frac{y^2}{4})^2y^2- 4b^2=4x_1x_2x_3x_4y^2=4 cy^2$ which yields
$$ y^{6} + 8 a y^{4} + 16 (a^{2} - 4 c) y^{2} - 64 b^{2} = 0. $$
That is to say, $2(x_1+x_2)$ is a root of $g(x)=0$. From the symmetry of $x_1,\ x_2,\ x_3,\ x_4$, we know that  $ 2(x_1+x_3),  2(x_1+x_4), 2(x_2+x_3), 2(x_2+x_4),2(x_3+x_4)$ are   roots of $g(x)=0$. In view of $g(x)=0$ has exactly 6 roots,  then all roots of $g(x)=0$ are
$2(x_1+x_2), 2(x_1+x_3),  2(x_1+x_4), 2(x_2+x_3), 2(x_2+x_4),2(x_3+x_4)$ which are all real.

(2)  If $x_1,\ x_2$ are positive, and  $x_3,\ x_4$ are negative, then $\tau(g)=2(x_1+x_2)$. Note that $\mathcal{E}(f)=x_1+x_2-x_3-x_4=2(x_1+x_2)$ implies that   $\tau(g)=\mathcal{E}(f)$.
\end{proof}

Now we are ready to determine how the energy changes of $K_{1,i,n-i-1}$ due to deleting one edge between $1$-part and $i$-part, where $i\in\{2,3,n-2\}$.
\begin{lemma}\label{lem:tth}  If $e$ is an edge between 1-part and $i$-part in $K_{1,i,n-i-1}$,  $i\in\{2,3,n-2\}$. Then
\begin{enumerate}[(1).]
\item  $ \mathcal{E}(K_{1,2,n-3}-e)<\mathcal{E}(K_{1,2,n-3}) $,
\item   $ \mathcal{E}(K_{1,3,n-4}-e)>\mathcal{E}(K_{1,3,n-4}) $,
\item     $ \mathcal{E}(K_{1, n-2,1}-e)>\mathcal{E}(K_{1, n-2,1})$.
\end{enumerate}
\end{lemma}
\begin{proof}
For short, we write $K_{1, i,n-i-1}$ as $G$, and  $n-i-1$ as $t$, where $i\in\{2,3,n-2\}$.
 Then
 $$
Q=\left(\begin{array}{rrr}
0 & i & t \\
1 & 0 &  t \\
1 & i  & 0
\end{array}\right),
$$ is a quotient matrix  of $G$,
and

\begin{equation}\label{quochart}
\Phi(Q,x)=x^3  - ( ti + i + t) x - 2 ti.
\end{equation}
  It is easy to see that  $\Phi(Q ,x)=0$ has     two negative  roots, say $-x_{1},\ -x_{2}$,  and one   positive root, say $x_{3}$. Then we have
$-x_{1}-x_{2}+x_{3}=0$ and $\mathcal{E}(G)=\mathcal{E}(Q)=2x_{3}$.  If we denote $g(x)=8\phi(Q, \frac{x}{2})=x^3  - 4(  t i + i + t) x - 16  t i$, then $\mathcal{E}(G)=\tau(g)$.

Similarly,
 $$Q'=\left(\begin{array}{rrrr}
0 & 0 & 0 & t \\
0 & 0 & i - 1 & t \\
0 & 1 & 0 & t \\
1 & 1 & i - 1 & 0
\end{array}\right),$$
is a quotient matrix  of $G-e$, and
\begin{equation}\label{quochart1}
  \Phi(Q',x)=x^4 + (-ti - i - t + 1) x^2 -2 (ti - t) x + ti - t.
\end{equation}
Obviously,  $\Phi(Q',x)=0$ has exactly two positive    roots. Applying   Lemma \ref{lem:roots} to $\Phi(Q',x)$, we can obtain  $h(x)=  x^6 - 8 ( ti +  t + i  - 1) x^4 + 16 \big((  ti + t)^2  + (i-1)^2 (2t +1) \big) x^2 - 256 (  ti - t)^2$, such that  $\mathcal{E}(G-e)=\mathcal{E}(Q')=\tau(h)$.

 Let
  \begin{align*}
 q(x)&=x^3 - 4 ((i + 1) t + i - 2) x+ 16 ti, \\
 r(x)&=h(x)-q(x)g(x)=- 16 [ (4 ti - 4 t - 1) x^2 - 8 ti x - 16 (2 i - 1) t^2 ].
 \end{align*}
Since $q(x)-g(x)=8x + 32$, $\tau(q)< \tau(g)$ and $\tau(g )=\tau(q g )$.

{\bf Case 1}  If $i=2 $.\\
Then
$$ 
\begin{aligned}
h(x)=&(x^3+ 4 x^2 - (12 t - 4) x -16 t)(x^3- 4 x^2- (12 t - 4) x + 16 t),\\
g(x)=& x^3- 4 (2 + 3 t) x -32 t.
\end{aligned}
$$
Suppose 
$$ 
\begin{aligned}
h_{1}(x)=& x^3+ 4 x^2 - (12 t - 4) x -16 t , \qquad  h_{2}(x)= x^3- 4 x^2- (12 t - 4) x + 16 t \\
r_{2}(x)=&\frac{1}{4}(h_{2}(x)-g(x))=-x^2+3x+12t.
\end{aligned}
$$

Then $h_{1}(x)-g(x)>0$ when $x>0$, so $\tau(h_{1})<\tau(g)$. Since $h_{2}(x)=(1-x)r_2(x)+x+4t$, $h_{2}(x)>0$ always holds  for any $x>\tau(r_2)>1$.   This means  all the positive roots of $h_{2}(x)$ are in the
interval $(0,\tau(r_2))$. On the other hand, we find that $h_{2}(x)>g(x)$ holds in $(0,\tau(r_2))$. Combining these with the fact $g(x)=0$ has exactly one positive root, we conclude  $\tau(g)>\tau(h_{2})$.  Consequently, $ \mathcal{E}(G)>\max\{\tau(h_{1}),\tau(h_{2})\}=\tau(h)=\mathcal{E}(G-e)$, Hence (1) holds.

{\bf Case 2} If $  i=3$.

Note that $g(2\sqrt{4t+3})=-48t<0 $, thus $\tau(g)>2\sqrt{4t+3}$. Since $r(x)=-16( (8 \, t - 1 ) x^{2} - 24 \, t x- 80 \, t^{2}) $ has only one positive root, say $x_0$, so $r(x)$ is a decreasing function for $x> x_0$. It is easy to see $x_0<2\sqrt{4t+3}$ when $t \geq 2$. Therefore, $h(\tau(g))=r(\tau(g))< r(2\sqrt{4t+3})<0$  for $t \geq 2$. Hence $h(\tau(g))=r(\tau(g))<0$, and then  $\tau(h)>\tau(g)$, i.e., $\mathcal{E}(G-e)>\mathcal{E}(G)$.

 {\bf Case 3} If $i=n-2$.

 Then $t=1$, and $g(x)={\left(x^{2} - 2 \, x- 8 \, i \right)} {\left(x + 2\right)}$, which yields that
$\mathcal{E}(G)=\tau(g)=1+\sqrt{1+8i}$.

Note that $h(1+\sqrt{1+8i})=r(1+\sqrt{1+8i})=32(-16 \, i^{2} + 36 \, i + 5 \, \sqrt{8 \, i + 1} - 3)$  is a decreasing function for $i \geq 3$. Hence, $h(1+\sqrt{1+8i})\leq h(6)=-448<0$  when $i \geq 3$. Consequently,   $\mathcal{E}(G-e)=\tau(h )>1+\sqrt{1+8i}=\mathcal{E}(G)$.
\end{proof}

The following proposition is an immediate result from Lemma \ref{lem:three} and Lemma \ref{lem:tth}.
\begin{proposition}\label{pro:tri}
\begin{enumerate}[(1).]
\item $\mathcal{E}(K_{1,2,n-3}-e)< \mathcal{E}(K_{1,2,n-3})$ for any edge $e$ between $1$-part and $2$-part.
\item  $\mathcal{E}(K_{1,i,n-i-1}-e)>\mathcal{E}(K_{1,i,n-i-1})$ for $i\geq 3$ and any edge $e$ between $1$-part and $i$-part.
\end{enumerate}
\end{proposition}

Combining these with the well-known results of bipartite graphs,   we can get our main result.

\begin{theorem}\label{thm:main} Let  $e$ be an edge between the $t_i$-part and $t_j$-part of $K_{t_1,\ldots,t_k}$.  Then
\begin{enumerate}[(1).]
\item  For $k \geq 4$,  if $t_i=t_j=1$, then  $\mathcal{E}(K_{t_1,\ldots,t_k}-e )<\mathcal{E}(K_{t_1,\ldots,t_k})$, otherwise,  $\mathcal{E}(K_{t_1,\ldots,t_k}-e)>\mathcal{E}(K_{t_1,\ldots,t_k})$.
\item For  $k=3$, if $t_i+t_j \leq 3$, then  $\mathcal{E}(K_{t_1,\ldots,t_k}-e)<\mathcal{E}(K_{t_1,\ldots,t_k})$, otherwise, $\mathcal{E}(K_{t_1,\ldots,t_k}-e)>\mathcal{E}(K_{t_1,\ldots,t_k})$.
\item For $k=2$, if $\min\{t_i,t_j\}=1$, then  $\mathcal{E}(K_{t_1,\ldots,t_k}-e)<\mathcal{E}(K_{t_1,\ldots,t_k})$, otherwise, $\mathcal{E}(K_{t_1,\ldots,t_k}-e)>\mathcal{E}(K_{t_1,\ldots,t_k})$.
\end{enumerate}
\end{theorem}


\begin{thebibliography}{99}
\bibitem{S. Akbari}
S. Akbari, E. Ghorbani, and M. Oboudi, Edge addition, singular values, and energy of graphs
and matrices, Linear Algebra Appl. 430 (2009) 2192-2199.

\bibitem{Gutman1}
I. Gutman, The energy of a graph, Ber. Math.-Statist. Sekt. Forsch. Graz 103 (1978)
1-22.

\bibitem{J. Day}
J. Day, and W. So, Graph energy change due to edge deletion, Linear Algebra Appl. 428 (2008) 2070-2078.



\bibitem{Gutman2}
I. Gutman, The energy of a graph: Old and new results, in: A. Betten, A. Kohnert, R.
Laue, A. Wassermann (Eds.), Algebraic Combinatorics and Applications, Springer-
Verlag, Berlin, 2001, 196-211.


\bibitem{Cvetkovic}
D. Cvetkovi\'c, P. Rowlinson, S. Simic,  An Introduction to the Theory of Graph Spectra, Cambridge University Press, Cambridge, 2010.


\bibitem{Dragan}
D. Stevanovi\'c, I. Gutman, M. U. Rehman, On spectral radius and energy of complete multipartite graphs, Ars Mathematica Contemporanea   9 (2015) 109-113.

\bibitem{sage}
W.A. Stein, et al., Sage Mathematics Software (Version 7.3), The Sage Development Team, http:// www.sagemath.org, 2016.

\end{thebibliography}
\end{document}